\documentclass[12pt,twoside]{amsart}
\usepackage{amsmath}
\usepackage{amsthm}
\usepackage{amsfonts}
\usepackage{amssymb}
\usepackage{latexsym}
\usepackage[all]{xy}

\date{}
\allowdisplaybreaks[4] \footskip=15pt
\renewcommand{\uppercasenonmath}[1]{}

\numberwithin{equation}{section} \theoremstyle{plain}
\newtheorem*{thm*}{Main Theorem}
\newtheorem{thm}{Theorem}[section]
\newtheorem{cor}[thm]{Corollary}
\newtheorem*{cor*}{Corollary}
\newtheorem{lem}[thm]{Lemma}
\newtheorem*{lem*}{Lemma}
\newtheorem{prop}[thm]{Proposition}
\newtheorem*{prop*}{Proposition}

\newtheorem*{rem*}{Remark}
\newtheorem{exa}[thm]{Example}
\newtheorem*{exa*}{Example}
\newtheorem{df}[thm]{Definition}
\newtheorem*{df*}{Definition}

\newtheorem*{conj*}{Conjecture}

\newtheorem*{Fa*}{Fact}
\newtheorem*{ack*}{ACKNOWLEDGEMENTS}

\newcommand{\pf}{\noindent\begin {proof}}
\newcommand{\epf}{\end{proof}}

\newcommand{\Ext}{\mbox{\rm Ext}}
\newcommand{\Hom}{\mbox{\rm Hom}}
\newcommand{\Tor}{\mbox{\rm Tor}}

\def\Im{\mathop{\rm Im}\nolimits}
\def\Ker{\mathop{\rm Ker}\nolimits}
\def\Coker{\mathop{\rm Coker}\nolimits}
\def\Tr{\mathop{\rm Tr}\nolimits}
\def\cTr{\mathop{\rm cTr}\nolimits}
\def\coOmega{\mathop{{\rm co}\Omega}\nolimits}
\def\mod{\mathop{\rm mod}\nolimits}
\def\Mod{\mathop{\rm Mod}\nolimits}
\def\cograde{\mathop{\rm cograde}\nolimits}

\def\id{\mathop{\rm id}\nolimits}

\def\inf{\mathop{\rm inf}\nolimits}
\def\add{\mathop{\rm add}\nolimits}
\def\Add{\mathop{\rm Add}\nolimits}

\def\Hom{\mathop{\rm Hom}\nolimits}
\def\Ext{\mathop{\rm Ext}\nolimits}

\def\lim{\mathop{\underrightarrow{\rm lim}}\nolimits}

\def\cograde{\mathop{\rm cograde}\nolimits}
\def\rank{\mathop{\rm rank}\nolimits}
\def\GInj{\mathop{\rm GInj}\nolimits}
\def\syz{\mathop{\rm syz}\nolimits}
\def\std{\mathop{\rm std}\nolimits}

\begin{document}
\begin{center}
{\large  \bf   Homological Aspects of the Dual Auslander Transpose}

\vspace{0.5cm}

Xi Tang\\
{\tiny{\it College of Science, Guilin University of Technology, Guilin 541004, Guangxi Province, P.R. China\\
E-mail: tx5259@sina.com.cn\\}}

\bigskip

Zhaoyong Huang\\
{\tiny{\it Department of Mathematics, Nanjing University, Nanjing 210093, Jiangsu Province, P.R. China\\
E-mail: huangzy@nju.edu.cn\\}}
\end{center}

\bigskip
\centerline { \bf  Abstract}
\bigskip
\leftskip10truemm \rightskip10truemm \noindent As a dual of the
Auslander transpose of modules, we introduce and study the
cotranspose of modules with respect to a semidualizing module $C$.
Then using it we introduce $n$-$C$-cotorsionfree modules, and show
that $n$-$C$-cotorsionfree modules possess many dual properties of
$n$-torsionfree modules. In particular, we show that
$n$-$C$-cotorsionfree modules are useful in characterizing the Bass
class and investigating the approximation theory for modules.
Moreover, we study $n$-cotorsionfree modules over artin algebras and
answer negatively an open question of Huang and Huang posed in 2012.
\vbox to 0.3cm{}
\bigskip
\section { \bf Introduction}
\bigskip
It is well known that the Auslander-Reiten theory plays a very
important role in representation theory of artin algebras and
homological algebra. One of the most powerful tools in this theory
is the Auslander transpose. With the aid of the Auslander transpose,
as a special case of $n$-syzygy modules over left and right noether
rings, Auslander and Bridger \cite{Au} introduced $n$-torsionfree
modules and obtained an approximation theory for finitely generated
modules when $n$-syzygy modules and $n$-torsionfree modules
coincide. Ever since then many authors have studied the homological
properties of these modules and related modules; see \cite{Au},
\cite{AR1}, \cite{AR2}, \cite{AR3}, \cite{HH}, \cite{HT}, \cite{JS},
\cite{R}, \cite{STMY}, \cite{STY}, \cite{T}, and so on. Based on
these references, two natural questions arise: (1) How to dualize
the Auslander transpose of modules appropriately? (2) Does the
notion of $n$-torsionfree modules have its dual as many notions in
classical homological algebra do? The aim of this paper is to study
these two questions, and we will define and investigate the
cotranspose of modules and $n$-cotorsionfree modules.

The paper is organized as follows.

In Section 2 we give some terminology and some preliminary results, and we also
introduce the notions of cotorsionless modules and coreflexive modules.

In Section 3 we introduce the cotranspose of
modules with respect to a semidualizing bimodule $C$, and
using it we introduce $n$-$C$-cotorsionfree modules as a dual of $n$-($C$-)torsionfree modules
in \cite{Au} and \cite{T}. We show that
$n$-$C$-cotorsionfree modules possess many dual properties of
$n$-($C$-)torsionfree modules. For example, we prove that a module is $n$-$C$-cotorsionfree
if and only if it admits some special proper resolutions of length at least $n$. Then, as an application,
we deduce that the Bass class with respect to $C$ coincides with the intersection of the class
of $\infty$-$C$-cotorsionfree modules and that of $\infty$-$C$-cospherical modules. As another application,
we get a dual version of the approximation theorem for
finitely generated modules over left and right noetherian rings in \cite[Proposition 2.21]{Au}
and its semiduazlizing version in \cite[Theorem A]{T}.

In Section 4 we generalize the cograde of finitely generated modules in \cite{O}
to general modules, and prove that for a ring $R$, the $n$-cosyzygy of a left $R$-module $M$ is
$n$-$C$-cotorsionfree if and only if the cograde of $\Ext_R^i(C,M)$ is at least $i-1$ for any $1\leq i\leq n$.
This result can be regarded as a dual version of \cite[Proposition 2.26]{Au}.

In Section 5, we focus on studying some special finitely generated
$n$-$C$-cotorsionfree modules (called $n$-cotorsionfree modules)
over artin algebras. In this case, we first show that the ordinary
Matlis duality induces a duality between the cotranspose (resp.
$n$-cotorsionfree modules) and the transpose (resp. $n$-torsionfree
modules). Then we obtain an equivalent characterization
when ($^{\bot}\mathcal{GI}, \mathcal{GI})$ forms a cotorsion pair,
where $\mathcal{GI}$ denotes the class of finitely generated
Gorenstein injective modules and $^{\bot}\mathcal{GI}$ is its left
orthogonal class. Finally, we give an example to illustrate that the
class of $\infty$-torsionfree modules is not closed under kernels of
epimorphisms in general. It answers negatively an open question of Huang
and Huang (\cite{HH}).

\section {\bf Preliminaries}
\bigskip

Throughout this paper, $R$ and $S$ are
fixed associative rings with unites. We use $\Mod R$ (resp.
$\Mod S^{op}$) to denote the class of left $R$-modules (resp. right
$S$-modules).

\begin{df} \label{df: 2.1} {\rm (\cite{HW}). An ($R$-$S$)-bimodule $_RC_S$ is called
\textit{semidualizing} if
\begin{enumerate}
\item[(a1)] $_RC$ admits a degreewise finite $R$-projective resolution.
\item[(a2)] $C_S$ admits a degreewise finite $S$-projective resolution.
\item[(b1)] The homothety map $_RR_R\stackrel{_R\gamma}{\rightarrow} \Hom_{S}(C,C)$ is an isomorphism.
\item[(b2)] The homothety map $_SS_S\stackrel{\gamma_S}{\rightarrow} \Hom_{R}(C,C)$ is an isomorphism.
\item[(c1)] $\Ext_{R}^{\geqslant 1}(C,C)=0.$
\item[(c2)] $\Ext_{S}^{\geqslant 1}(C,C)=0.$
\end{enumerate}}
\end{df}

From now on, $_RC_S$ is a semidualizing bimodule. We write $(-)^*=\Hom(-,C)$ and $(-)_*=\Hom(C,-)$.
For a module $M\in \Mod R$, we have the following two canonical valuation homomorphisms:
$$\sigma_M: M\rightarrow M^{**}$$ defined by
$\sigma_M(x)(f)=f(x)$ for any $x\in M$ and $f\in M^*$, and
$$\theta_M:
C\otimes_SM_*\rightarrow M$$ defined by $\theta_M(x\otimes
f)=f(x)$ for any $x\in C$ and $f\in M_*$.

\begin{df} \label{df: 2.2} {\rm (\cite{HW}). The \textit{Bass class} $\mathcal {B}_C(R)$
with respect to $C$ consists of all left $R$-modules $M$ satisfying
\begin{enumerate}
\item[(B1)] $\Ext^{\geqslant 1}_R(C,M)=0,$
\item[(B2)] $\Tor_{\geqslant 1}^S(C,\Hom_R(C,M))=0$, and
\item[(B3)] $\theta_M$ is an isomorphism in $\Mod R$.
\end{enumerate}}
\end{df}

Let $M$ be a finitely presented left $R$-module and $$P_1 \buildrel
{f_0} \over \to P_0 \to M \to 0$$ a finitely generated projective
presentation of $M$. Then $\Tr_CM:=\Coker {f_0}^{*}$ is called the
\textit{(Auslander) transpose with respect to $C$} (\cite{HT}). When $R=S$ and
$_RC_S={_RR_R}$, the Auslander transpose with respect to $C$ is
just the \textit{Auslander transpose} (\cite{Au}).

\begin{prop} \label{prop: 2.3} {\rm (\cite[Proposition 2.6]{Au} and \cite[Lemma 2.1]{HT}).}
Let $M$ be a finitely presented left $R$-module. Then there exists an exact sequence:
$$0\rightarrow \Ext_{S}^1(\Tr_CM,C)\rightarrow M\buildrel {\sigma_M}\over \longrightarrow
M^{**}\rightarrow \Ext_{S}^2(\Tr_CM,C)\rightarrow 0.$$
\end{prop}

Recall that a module $M\in \Mod R$ is called \emph{$C$-torsionless}
if $\sigma_M$ is a monomorphism, and $M$ is called
\emph{$C$-reflexive} if $\sigma_M$ is an isomorphism. As the duals
of $C$-torsionless modules and $C$-reflexive modules, we introduce
the following

\begin{df} \label{df: 2.4} {\rm A module $M\in \Mod R$ is called
\emph{$C$-cotorsionless} if $\theta_M$ is an epimorphism, and $M$ is
called \emph{$C$-coreflexive} if $\theta_M$ is an isomorphism.}
\end{df}

For a module $M\in \Mod R$, we denote by $\Add_RM$ the subclass of $\Mod R$
consisting of all direct summands of direct sums of copies
of $M$.

\begin{lem} \label{lem: 2.5} The following statements hold.
\begin{enumerate}
\item For any $W\in \Add_R C$, $W$ is $C$-coreflexive, $W_*$ is a projective left $S$-module
and $\Ext_R^{\geq 1}(C,W)=0$.
\item For any injective left $R$-module $I$, $I$ is $C$-coreflexive and $\Tor_{\geqslant 1}^S(C,I_*)=0$.
\end{enumerate}
\end{lem}

\begin{proof} (1) follows from \cite[Lemma 5.1(b)]{HW}, and (2) follows from \cite[Lemma 5.1(c)]{HW}.
\end{proof}

\begin{df} \label{df: 2.6}
{\rm  (\cite{SSW}) Let $\mathcal{X}$ be a subclass of $\Mod R$.

(1) An exact sequence $\mathbb{E}$ in $\Mod R$ is
called \emph{$\Hom_R(\mathcal{X},-)$-exact} (resp. \emph{$\Hom_R(-,\mathcal{X})$-exact}) if $\Hom_R(X,\mathbb{E})$
(resp. $\Hom_R(\mathbb{E},X)$) is exact for any $X\in \mathcal{X}$.

(2) An exact sequence
$${\bf X}:=\cdots \rightarrow X_1\rightarrow X_0\rightarrow X^{0}\rightarrow
X^{1}\rightarrow \cdots$$ in $\Mod R$ with $X_i, X^i \in \mathcal{X}$
is called \textit{totally $\mathcal{X}$-acyclic} if it is
$\Hom_{R}(\mathcal{X},-)$-exact and $\Hom_{R}(-,\mathcal{X})$-exact.}
\end{df}

\begin{df} \label{df: 2.7}
{\rm (\cite{EJ}) A module $M\in \Mod R$ is
called \textit{Gorenstein injective}, if there exists a
totally acyclic complex of injective modules
$${\bf I}:=\cdots \rightarrow I_1\rightarrow I_0\rightarrow I^{0}\rightarrow
I^{1}\rightarrow \cdots$$ in $\Mod R$ such that $M\cong\Im (I_0\rightarrow
I^{0})$.}
\end{df}

\section {\bf The cotranspose and $n$-$C$-cotorsionfree modules}
\bigskip

In this section, we introduce and study the cotranspose of modules and $n$-cotorsionfree modules
with respect to the given semidualizing bimodule $_RC_S$.

Let $M\in \Mod R$. We use
$$0\to M \to I^0(M) \buildrel {f^0} \over \to I^1(M) \buildrel
{f^1}\over \to \cdots \buildrel {f^{i-1}} \over \to I^i(M)\buildrel
{f^{i}} \over \to \cdots \eqno{(3.1)}$$ to denote a minimal
injective resolution of $M$ in $\Mod R$. For any $n\geq 1$,
$\coOmega^n(M):= \Im f^{n-1}$ is called the \emph{$n$-th cosyzygy}
of $M$, and in particular, put $\coOmega^{0}(M)=M$. A module in
$\Mod R$ is called \emph{$n$-cosyzygy} if it is isomorphic to the
$n$-th cosyzygy of some module in $\Mod R$.
We introduce the dual notion of the Auslander transpose of modules as follows.

\begin{df} \label{df: 3.1} {\rm For a module $M\in \Mod R$, $\cTr_C M:=\Coker {f^0}_*$
is called the \emph{cotranspose} of $M$ with respect to $_RC_S$.}
\end{df}

The following result is a dual version of Proposition 2.3.

\begin{prop} \label{prop: 3.2} Let $M\in \Mod R$. Then there exists an exact sequence:
{\footnotesize $$0\rightarrow \Tor^S_2(C,\cTr_C M)\rightarrow
C\otimes_SM_*\buildrel {\theta_M}\over \longrightarrow
M\rightarrow \Tor^S_1(C,\cTr_C M)\rightarrow 0.$$}
\end{prop}

\begin{proof} By applying the functor $(-)_*$ to the minimal injective resolution (3.1) of $M$,
We get an exact sequence: $$0\rightarrow M_*\rightarrow
{I^0(M)}_*\stackrel{{f^0}_*}{\longrightarrow}{I^1(M)}_*\rightarrow
\cTr_C M\rightarrow 0$$ in $\Mod S$. Let $f^0=\alpha\cdot\pi$ (where $\pi:
I^0(M)\twoheadrightarrow \Im f^0$ and $\alpha: \Im
f^0\rightarrowtail I^1(M)$) and ${f^0}_*=\alpha^{'}\cdot\pi^{'}$
(where $\pi^{'}: I^0(M)_*\twoheadrightarrow \Im {f^0}_*$ and
$\alpha^{'}: \Im {f^0}_*\rightarrowtail I^1(M)_*$) be the natural
epic-monic decompositions of $f^0$ and ${f^0}_*$ respectively. Since
$\Tor_{1}^S(C,{I^0(M)}_*)=0$ and $\theta_{I^0(M)}$ is an
isomorphism by Lemma 2.5(2), we have the following commutative diagram
with exact rows: {\tiny
$$\xymatrix{0 \ar[r]& \Tor^S_1(C,\Im {f^0}_*) \ar[r] &
C\otimes_RM_{*} \ar[r] \ar[d]^{\theta_M}&
C\otimes_S{I^0(M)}_{*} \ar[r]^{1_{C}\otimes \pi^{'}}
\ar[d]^{\theta_{I^0(M)}}&
C\otimes_S\Im {f^0}_* \ar[r] \ar@{-->}[d]^{h}& 0\\
& 0 \ar[r]&  M \ar[r] & I^0(M) \ar[r]^{\pi} & \Im f^0 \ar[r]& 0,}$$}
where $h$ is an induced homomorphism. Then
$\pi\cdot\theta_{I^0(M)}=h\cdot(1_{C}\otimes \pi^{'})$. In addition,
by the snake lemma, we have $\Ker \theta_M\cong \Tor^S_1(C,\Im
{f^0}_*)$ and $\Coker \theta_M\cong\Ker h$.

On the other hand, since $\Tor_{1}^S(C,{I^1(M)}_*)=0$ by
Lemma 2.5(2), by applying the functor $C\otimes_S-$ to the exact
sequence:
$$0\rightarrow \Im {f^0}_* \stackrel{\alpha^{'}}{\rightarrow} {I^1(M)}_{*}\rightarrow
\cTr_C M\rightarrow 0,$$ we get the following exact sequence:
{\footnotesize
$$0\rightarrow \Tor^S_1(C,\cTr_C M)\rightarrow
C\otimes_S\Im {f^0}_*\stackrel{1_{C}\otimes
\alpha^{'}}{\longrightarrow}C\otimes_S{I^1(M)}_{*}\rightarrow
C\otimes_S\cTr_C M\rightarrow 0$$} and the isomorphism:
$$\Tor^S_1(C,\Im
{f^0}_*)\cong \Tor^S_2(C,\cTr_C M).$$ Because
$$\xymatrix{C\otimes_S {I^0(M)}_{*} \ar[r]^{1_{C}\otimes
{f^0}_*}
\ar[d]^{\theta_{I^0(M)}} & C\otimes_S {I^1(M)}_{*} \ar[d]^{\theta_{I^1(M)}} \\
I^0(M) \ar[r]^{f^0} & I^1(M)}$$ is a commutative diagram,
$f^0\cdot\theta_{I^0(M)}=\theta_{I^1(M)}\cdot(1_{C}\otimes
{f^0}_*)$. Because ${f^0}_*=\alpha^{'}\cdot\pi^{'}$, $1_{C}\otimes
{f^0}_*=1_{C}\otimes (\alpha^{'}\cdot \pi^{'})=(1_{C}\otimes
\alpha^{'})\cdot (1_{C}\otimes \pi^{'})$. Thus we have $\alpha\cdot
h\cdot(1_{C}\otimes
\pi^{'})=\alpha\cdot\pi\cdot\theta_{I^0(M)}=f^0\cdot\theta_{I^0(M)}
=\theta_{I^1(M)}\cdot(1_{C}\otimes
{f^0}_*)=\theta_{I^1(M)}\cdot(1_{C}\otimes \alpha^{'})\cdot
(1_{C}\otimes \pi^{'})$. Because $1_{C}\otimes \pi^{'}$ is epic,
$\alpha\cdot h=\theta_{I^1(M)}\cdot(1_{C}\otimes \alpha^{'})$.
Notice that $\alpha$ is monic and $\theta_{I^1(M)}$ is an
isomorphism (by Lemma 2.5(2)), so $\Coker \theta_M\cong\Ker h\cong\Ker
(1_{C}\otimes \alpha^{'})\cong \Tor^S_1(C,\cTr_C M)$. Consequently
we obtain the desired exact sequence.
\end{proof}

For any $n\geq 1$, recall from \cite{T} that a finitely presented
left $R$-module $M$ is called \emph{$n$-$C$-torsionfree} if $\Ext_S^i(\Tr_C
M,C)=0$ for any $1\leq i\leq n$. When $R=S$ and $_RC_S={_RR_R}$, an
$n$-$C$-torsionfree module is just an \textit{$n$-torsionfree module} (\cite{Au}).
We introduce the dual notion of $n$-$C$-torsionfree modules as follows.

\begin{df} \label{df: 3.3}
{\rm  Let $M\in \Mod R$ and $n\geq 1$. Then $M$ is called
\emph{$n$-$C$-cotorsionfree} if $\Tor_i^S(C,\cTr_C M)=0$ for any
$1\leq i\leq n$; and $M$ is called \emph{$\infty$-$C$-cotorsionfree}
if it is $n$-$C$-cotorsionfree for all $n$. In particular, every
left $R$-module is 0-$C$-cotorsionfree.}
\end{df}

It is trivial that a left $R$-module is $n$-$C$-cotorsionfree if it
is $m$-cotorsionfree for some $m\geq n$. It is easy to verify that
the class of $n$-$C$-cotorsionfree $R$-modules is closed under
direct summands and finite direct sums.

Note that for any $M\in \Mod R$, there exists an exact sequence:
$$0\rightarrow M_*\rightarrow
{I^0(M)}_*\stackrel{{f^0}_*}{\longrightarrow}{I^1(M)}_*\rightarrow
\cTr_C M\rightarrow 0.$$ The following corollary is an immediate
consequence of Proposition 3.2.

\begin{cor} \label{cor: 3.4} Let $M\in \Mod R$.
Then we have
\begin{enumerate}
\item $M$ is 1-$C$-cotorsionfree if and only if it is $C$-cotorsionless.
\item $M$ is 2-$C$-cotorsionfree if and only if
it is $C$-coreflexive.
\item For any $n\geq 3$,
$M$ is $n$-$C$-cotorsionfree if and only if it is $C$-coreflexive and $\Tor_i^S(C,M_{*})=0$ for any $1\leq i\leq n-2$.
\end{enumerate}
\end{cor}

\begin{prop} \label{prop: 3.5} Let $0\rightarrow L\rightarrow M\rightarrow N\rightarrow 0$
be a $\Hom_R(C,-)$-exact exact sequence in $\Mod R$ with $L$
$n$-$C$-cotorsionfree. Then $M$ is $n$-$C$-cotorsionfree if and only
if so is $N$.
\end{prop}

\begin{proof}  By assumption we
have an exact sequence: $$0\rightarrow
L_*\rightarrow M_*\rightarrow N_*\rightarrow 0$$ in $\Mod S$. Then we get the
following commutative diagram with exact rows:
$$\xymatrix{&C\otimes_SL_{*} \ar[d]^{\theta_L} \ar[r]
& C\otimes_SM_{*} \ar[r] \ar[d]^{\theta_M}
& C\otimes_SN_{*} \ar[d]^{\theta_N} \ar[r] & 0\\
 0 \ar[r]& L \ar[r] & M \ar[r] & N \ar[r]& 0}$$
and the following exact sequence:
$$\Tor_i^S(C,{L}_{*})\rightarrow
\Tor_i^S(C,{M}_{*})\rightarrow \Tor_i^S(C,{N}_{*})\rightarrow
\Tor_{i-1}^S(C,{L}_{*})$$ for any $i\geq 2$. Now the assertion
follows easily from the snake lemma and Corollary 3.4.
\end{proof}

Let $\mathcal{X}$ be a subclass of $\Mod R$ and $M\in \Mod R$.
Following Enochs and Jenda \cite{EJ}, a homomorphism $\phi:
X\rightarrow M$ in $\Mod R$ with $X\in\mathcal{X}$ is called
a \emph{$\mathcal{X}$-precover} of $M$ if $\Hom_R(X^{'}, \phi):
\Hom_R(X^{'}, X)\rightarrow \Hom_R(X^{'}, M)$ is epic for any
$X^{'}\in \mathcal{X}$. A $\mathcal{X}$-precover $\phi: X\rightarrow
M$ is called a \emph{$\mathcal{X}$-cover} if every endomorphism  $g:
X\rightarrow X$ such that $\phi g=\phi$ is an isomorphism. Dually
the notion of an \emph{$\mathcal{X}$-(pre)envelope} of $M$ is
defined. Recall from \cite{H} that an exact
sequence (of finite or infinite length):
$$\cdots \to X_n \to \cdots \to X_1 \to X_0 \to M \to 0$$
in $\Mod R$ is called an \emph{$\mathcal{X}$-resolution} of
$M$ if each $X_i\in \mathcal{X}$. Furthermore, such an
$\mathcal{X}$-resolution is called \emph{proper} if $X_i\twoheadrightarrow
\Im(X_i\to X_{i-1})$ is an $\mathcal{X}$-precover of $\Im(X_i\to
X_{i-1})$ (note: $X_{-1}=M$). Dually, the notion of an \textit{$\mathcal {X}$-coresolution} of $M$ is defined.
The \textit{$\mathcal {X}$-injective
dimension} $\mathcal {X}$-$\id_R(M)$ of $M$ is defined as
$\inf\{n\mid$ there exists an $\mathcal
{X}$-coresolution $0 \to M \to X^0 \to X^1 \to \cdots \to X^n \to 0$
of $M$ in $\Mod R\}$.

In the following result we give an equivalent characterization of
$n$-$C$-cotorsionfree modules in terms of proper $\Add_R C$-resolutions
of modules. It is dual to \cite[Corollary 3.3]{T}.

\begin{prop} \label{prop: 3.6} Let $M\in\Mod R$ and $n\geq 1$.
Then $M$ is $n$-$C$-cotorsionfree if and only if there exists a
proper $\Add_R C$-resolution $W_{n-1}\rightarrow \cdots \rightarrow
W_1\rightarrow W_0\rightarrow M\rightarrow 0$ of $M$ in $\Mod R$.
\end{prop}

\begin{proof} We proceed by induction on $n$.

Let $n=1$ and $M$ be 1-$C$-cotorsionfree. Then $\theta_M$ is epic by Corollary 3.4.
Since  there exists an epimorphism $S^{(X)}\twoheadrightarrow
M_{*}$. So we get an epimorphism $C^{(X)}\twoheadrightarrow
C\otimes_SM_{*}$, which induces an epimorphism
$C^{(X)}\twoheadrightarrow M$ because $\theta_M$ is epic. By
\cite[Proposition 5.3]{HW}, every module in $\Mod R$ admits an
$\Add_RC$-precover. It follows that $M$ admits an epic $\Add_R
C$-precover. Conversely, let $W_0\twoheadrightarrow M$ be an epic $\Add_R
C$-precover of $M$. Because $\theta_{W_0}$ is an isomorphism by
Lemma 2.5(2), from the following commutative diagram with exact rows:
$$\xymatrix{C\otimes_S{W_0}_*
\ar[r] \ar[d]^{\theta_{W_0}}& C\otimes_SM_* \ar[d]^{\theta_M} \ar[r] & 0\\
W_0 \ar[r] &  M \ar[r]& 0}$$ we get that $\theta_M$ is epic and $M$
is 1-$C$-cotorsionfree.

Let $n=2$ and $M$ be 2-$C$-cotorsionfree. By the above argument, there exists an
exact sequence $0\rightarrow N\rightarrow W_0\rightarrow
M\rightarrow 0$ in $\Mod R$ with $W_0\twoheadrightarrow M$
an $\Add_R C$-precover of $M$. Then we have the following commutative
diagram with exact rows:
$$\xymatrix{& C\otimes_SN_{*} \ar[r] \ar[d]^{\theta_{N}}& C\otimes_S{W_0}_*
\ar[r] \ar[d]^{\theta_{W_0}}& C\otimes_SM_{*} \ar[r] \ar[d]^{\theta_{M}}&0 \\
0\ar[r]& N \ar[r] & W_0\ar[r]& M \ar[r] & 0.}$$ Because both
$\theta_{W_0}$ and $\theta_{M}$ are isomorphisms by Lemma 2.5(1) and
Corollary 3.4(2), $\theta_{N}$ is epic by the snake lemma, and hence
$N$ is 1-$C$-cotorsionfree by Corollary 3.4(1). It follows from the above argument
that $N$ admits an epic $\Add_R C$-precover $W_1\twoheadrightarrow
N$. Then the spliced sequence $W_1\rightarrow W_0\rightarrow
M\rightarrow 0$ is as desired. Conversely, let $W_1\rightarrow W_0\rightarrow M\rightarrow 0$ be a
proper $\Add_R C$-resolution of $M$. Put $N=\Ker(W_0\rightarrow M)$.
Then $N$ is 1-$C$-cotorsionfree by the above argument, and so $\theta_N$ is epic by
Corollary 3.4(1). Now the commutative diagram above implies that
$\theta_{M}$ is an isomorphism. Thus $M$ is 2-$C$-cotorsionfree by
Corollary 3.4(2).

Now suppose that $n\geq 3$ and $M$ is $n$-$C$-cotorsionfree. Then $\theta_M$ is an isomorphism
and $\Tor_i^S(C,M_{*})$ $=0$ for any $1\leq i\leq n-2$ by Corollary
3.4(3). In addition, by the induction hypothesis there exists an exact sequence
$0\rightarrow N\rightarrow W_0\rightarrow M\rightarrow 0$ in $\Mod
R$ with $W_0\in \Add_R C$ such that $0\rightarrow N_{*}\rightarrow
{W_0}_{*}\rightarrow M_{*}\rightarrow 0$ is also exact with
${W_0}_{*}$ projective. Then $\Tor^S_{i}(C,N_{*})\cong
\Tor^S_{i+1}(C,M_{*})=0$ for $1\leq i \leq n-3$, and we have the
following commutative diagram with exact rows:
$$\xymatrix{0\ar[r]&C\otimes_SN_{*} \ar[d]^{\theta_{N}}
\ar[r] & C\otimes_S{W_0}_{*} \ar[r] \ar[d]^{\theta_{W_0}}
& C\otimes_SM_{*} \ar[d]^{\theta_M} \ar[r] & 0\\
0 \ar[r]& N \ar[r] & W_0 \ar[r] & M \ar[r]& 0.}$$ Because
$\theta_{W_0}$ is an isomorphism by Lemma 2.5(1), $\theta_N$ is also an
isomorphism. Thus $N$ is $(n-1)$-$C$-cotorsionfree by Corollary 3.4(3)
and therefore the assertion follows from the induction hypothesis.

Conversely, assume that there exists a proper $\Add_R C$-resolution
$W_{n-1}\rightarrow \cdots \rightarrow W_1\rightarrow W_0\rightarrow
M\rightarrow 0$ of $M$ in $\Mod R$. Put $N=\Im(W_1\to W_0)$. Then
$0\rightarrow N_{*}\rightarrow {W_0}_{*}\rightarrow M_{*}\rightarrow
0$ is exact with ${W_0}_{*}$ projective. Because $N$ is
$(n-1)$-$C$-cotorsionfree by the induction hypothesis, $\theta_N$ is
an isomorphism and $\Tor_i^S(C,N_{*})=0$ for any $1\leq i\leq n-3$
by Corollary 3.4(3).

Consider the following commutative diagram with exact rows:
$$\xymatrix{&C\otimes_SN_{*} \ar[d]^{\theta_{N}}
\ar[r] & C\otimes_S{W_0}_{*} \ar[r] \ar[d]^{\theta_{W_0}}
& C\otimes_SM_{*} \ar[d]^{\theta_M} \ar[r] & 0\\
0 \ar[r]& N \ar[r] & W_0 \ar[r] & M \ar[r]& 0.}$$ Because
$\theta_{W_0}$ is an isomorphism by Lemma 2.5(1), $\theta_{M}$ is an
isomorphism and $0\to C\otimes_SN_{*} \to C\otimes_S{W_0}_{*}
\to C\otimes_SM_{*} \to 0$ is exact. So $\Tor_1^S(C,M_{*})=0$ and
$\Tor_{i+1}^S(C,M_{*})\cong \Tor_i^S(C,N_{*})=0$ for any $1\leq
i\leq n-3$, that is, $\Tor_{i}^S(C,M_{*})=0$ for any $1\leq i\leq
n-2$. Thus $M$ is $n$-$C$-cotorsionfree by Corollary 3.4(3).
\end{proof}

As an immediate consequence of Proposition 3.6 we have the following

\begin{cor} \label{cor: 3.7} For a module $M\in\Mod R$, the following statements are equivalent.
\begin{enumerate}
\item $M$ is 1-$C$-cotorsionfree (that is, $M$ is $C$-cotorsionless).
\item There exists an exact sequence $0\rightarrow N\rightarrow
W\rightarrow M\rightarrow 0$ in $\Mod R$ with $W\in\Add_R C$
and $\Ext_R^1(C,N)=0$.
\item There exists an
epimorphism $W\twoheadrightarrow M$ in $\Mod R$ with $W\in\Add_RC$.
\end{enumerate}
\end{cor}

It follows from Proposition 3.6 that a module $M\in \Mod R$ is
$\infty$-$C$-cotorsionfree if and only if $M$ has an exact proper
$\Add_R C$-resolution $\cdots \rightarrow W_2\rightarrow
W_1\rightarrow W_0\rightarrow M\rightarrow 0$ in $\Mod R$. A module
$M\in \Mod R$ is called \textit{$n$-$C$-cospherical} if
$\Ext_R^i(C,M)=0$ for and $1\leq i\leq n$, and $M$ is called
\emph{$\infty$-$C$-cospherical} if it is $n$-$C$-cospherical for all
$n$. The following result shows that the Bass class with respect to $C$ coincides with the intersection of the class
of $\infty$-$C$-cotorsionfree modules and that of $\infty$-$C$-cospherical modules.

\begin{thm} \label{thm: 3.8}  For a module $M\in\Mod R$, the following statements are equivalent.
\begin{enumerate}
\item $M$ is $\infty$-$C$-cotorsionfree and
$\infty$-$C$-cospherical.
\item $M\in \mathcal {B}_C(R)$.
\end{enumerate}
\end{thm}

\begin{proof} By Proposition 3.6 and
\cite[Theorem 6.1]{HW}.
\end{proof}

Auslander and Bridger obtained in \cite[Proposition 2.21]{Au} an approximation theorem for
finitely generated modules over left and right noetherian rings.
Takahashi in \cite[Theorem A]{T} got a semiduazlizing version of this result. We dualize \cite[Theorem A]{T}
as follows.

\begin{thm} \label{thm: 3.9} Let $M\in \Mod R$ and $n\geq 1$.
Then the following statements are equivalent.
\begin{enumerate}
\item $\coOmega^n(M)$ is $n$-$C$-cotorsionfree.
\item There exists an exact sequence $0\rightarrow M\rightarrow
X\rightarrow Y\rightarrow 0$ in $\Mod R$ such that $X$ is
$n$-$C$-cospherical and $\Add_RC$-$\id_RY\leq n-1$.
\end{enumerate}
\end{thm}
\begin{proof}
(1) $\Rightarrow$ (2). By Proposition 3.6 and Corollary 3.7, the
fact that $\coOmega^n(M)$ is $n$-$C$-cotorsionfree implies that
there exists an exact sequence $0\rightarrow N_0\rightarrow
W_0\rightarrow \coOmega^n(M)\rightarrow 0$ in $\Mod R$ with $W_0\in
\Add_RC$, $N_0$ ($n$-1)-$C$-cotorsionfree and $\Ext_R^1(C,N_0)$
\linebreak $=0$.
We get the following pullback diagram: $$\xymatrix{ & &  0 \ar[d] & 0 \ar[d] & \\
& & \coOmega^{n-1}(M) \ar@{=}[r] \ar[d]&\coOmega^{n-1}(M)   \ar[d]& \\
0 \ar[r]&  N_0 \ar[r] \ar@{=}[d]& X_0 \ar[r] \ar[d]& I^{n-1}(M) \ar[r] \ar[d] & 0 \\
0 \ar[r]& N_0 \ar[r] & W_0 \ar[r] \ar[d]& \coOmega^n(M) \ar[r]  \ar[d]& 0 \\
& & 0 & 0. & }$$ If $n=1$, then the middle column is the desired
sequence.

Let $n\geq$ 2. Since $I^{n-1}(M)\in$ $\mathcal {B}_C(R)$,
$I^{n-1}(M)$ is $\infty$-$C$-cotorsionfree by Theorem 3.8. Note that
$N_0$ is ($n$-1)-$C$-cotorsionfree and $\Ext_R^1(C,N_0)=0$. By
Proposition 3.5, $X_0$ is ($n$-1)-$C$-cotorsionfree. Thus there
exists an exact sequence $0\rightarrow Z_0\rightarrow U_0\rightarrow
X_0\rightarrow 0$ in $\Mod R$ with $U_0\in $ $\Add_RC$, $Z_0$
$(n-2)$-$C$-cotorsionfree and $\Ext_R^1(C,Z_0)=0$ by Proposition 3.6.
We construct the following pullback diagram:
$$\xymatrix{  &  0 \ar[d] & 0 \ar[d] & \\
 & Z_0 \ar@{=}[r] \ar[d]& Z_0   \ar[d]& \\
0 \ar[r]&  Y_0 \ar[r] \ar[d]& U_0 \ar[r] \ar[d]& W_{0} \ar[r] \ar@{=}[d] & 0 \\
0 \ar[r]& \coOmega^{n-1}(M) \ar[r] \ar[d]& X_0 \ar[r] \ar[d]& W_0 \ar[r]  & 0 \\
& 0 & 0 & }$$ such that $\Add_RC$-$\id_R Y_0\leq 1$ and
$\Ext_R^i(C,Z_0)=0$ for $i=1,2$ because $\Ext_R^1(C,X_0)=0$. Using
the leftmost column in this diagram, we also
have the following pullback diagram: $$\xymatrix{ & &  0 \ar[d] & 0 \ar[d] & \\
& & \coOmega^{n-2}(M) \ar@{=}[r] \ar[d]& \coOmega^{n-2}(M) \ar[d]& \\
0 \ar[r]&  Z_0 \ar[r] \ar@{=}[d]& X_1 \ar[r] \ar[d]& I^{n-2}(M) \ar[r] \ar[d] & 0 \\
0 \ar[r]& Z_0 \ar[r] & Y_0 \ar[r] \ar[d]& \coOmega^{n-1}(M) \ar[r]  \ar[d]& 0 \\
& & 0 & 0. & }$$ It follows from the middle row in the above
diagram that $\Ext_R^i(C,X_1)=0$ for $i=1,2$. Therefore,
if $n=2$, then the middle column in the above diagram is the
desired exact sequence.

Let $n \geq 3$. Since $Z_0$ is ($n$-2)-$C$-cotorsionfree and
$\Ext_R^1(C,Z_0)=0$, $X_1$ is ($n$-2)-$C$-cotorsionfree by
Proposition 3.5. We have an exact sequence $0\rightarrow
Z_1\rightarrow U_1\rightarrow X_1\rightarrow 0$ in $\Mod R$ with
$U_1\in \Add_RC$, $Z_1$ $(n-3)$-$C$-cotorsionfree and
$\Ext_R^1(C,Z_1)=0$ by Proposition 3.6 again.
Iterating the above construction of pullback diagrams, we eventually
obtain the desired exact sequence.

(2) $\Rightarrow$ (1). Since $\Add_RC$-$\id_R Y\leq n-1$, there exists an
exact sequence $0\rightarrow
Y\stackrel{d_0}{\rightarrow}W^0\stackrel{d_1}{\rightarrow}W^1\rightarrow
\cdots \stackrel{d_{n-1}}{\rightarrow} W^{n-1}\rightarrow 0$ in $\Mod R$ with all $W^i\in \Add_RC$. Set
$Y_i=\Im d_i$ for each $i$. We have the following pushout diagram:
$$\xymatrix{  &  0 \ar[d] & 0 \ar[d] & \\
0\ar[r] & M\ar[r] \ar[d]& I^0(M) \ar[r] \ar[d]& \coOmega^1(M) \ar[r] \ar@{=}[d]& 0\\
0 \ar[r]&  X \ar[r] \ar[d]& H_0 \ar[r] \ar[d]& \coOmega^1(M) \ar[r] & 0 \\
&Y\ar@{=}[r] \ar[d]&Y \ar[d]\\
& 0 & 0. & }$$ Then $H_0\cong
Y\oplus I^0(M)$. Adding $I^0(M)$ to the exact sequence $0\rightarrow Y\rightarrow W^0\rightarrow Y_1\rightarrow
0$, we get an exact sequence $0\rightarrow Y\oplus I^0(M)\rightarrow W^0\oplus I^0(M)\rightarrow Y_1\rightarrow
0$. Thus the following two pushout diagrams are obtained.
$$\xymatrix{&  &  0 \ar[d] & 0 \ar[d] & \\
0\ar[r] & X\ar[r] \ar@{=}[d]& Y\oplus I^0(M) \ar[r] \ar[d]& \coOmega^1(M)  \ar[r] \ar[d]& 0\\
0\ar[r] & X\ar[r] & W^0\oplus I^0(M) \ar[r] \ar[d]& X_1  \ar[r] \ar[d]& 0\\
& & Y_1 \ar@{=}[r] \ar[d]&Y_1 \ar[d]\\
& & 0 & 0& }$$ and
$$\xymatrix{  &  0 \ar[d] & 0 \ar[d] & \\
0\ar[r] & \coOmega^1(M)\ar[r] \ar[d]& I^1(M) \ar[r] \ar[d]& \coOmega^2(M)  \ar[r] \ar@{=}[d]& 0\\
0 \ar[r]&  X_1 \ar[r] \ar[d]& H_1 \ar[r] \ar[d]& \coOmega^2(M) \ar[r] & 0 \\
&Y_1\ar@{=}[r] \ar[d]&Y_1 \ar[d]\\
& 0 & 0. & }$$ Repeating the procedure in this way yields the following exact sequence:
$$0\rightarrow X_i\rightarrow W^i\oplus I^i(M)\rightarrow X_{i+1}\rightarrow 0$$
for any $0\leq i\leq n-1$, where $X_0=X$. Since $\Ext_R^i(C,X_0)=0$ for any $1\leq i\leq n$ by assumption,
$\Ext_R^j(C,X_i)=0$ for any $1\leq j\leq n-i$. Then there exists an exact sequence:
$$0\rightarrow {X_i}_*\rightarrow (W^i\oplus I^i(M))_*\rightarrow {X_{i+1}}_*\rightarrow
0$$ for any $0\leq i\leq n-1$. By Lemma 2.5, each $\theta_{W^i\oplus I^i(M)}$ is an isomorphism.
Now we have the following commutative diagram with exact rows:
$$\xymatrix{C\otimes_S(W^0\oplus I^0(M))_* \ar[r] \ar[d]^{\theta_{W^0\oplus I^0(M)}}
& C\otimes_S{X_1}_* \ar[d]^{\theta_{X_1}}\ar[r] & 0\\
W^0\oplus I^0(M) \ar[r] &  X_1 \ar[r]& 0.}$$
It follows that $\theta_{X_1}$ is epic and so $X_1$ is 1-$C$-cotorsionfree by Corollary 3.4(1).
Also, there exists the following commutative diagram with exact rows:
$$\xymatrix{ & C\otimes_S{X_1}_{*} \ar[r] \ar[d]^{\theta_{X_1}} & C\otimes_S(W^1\oplus I^1(M))_{*}
\ar[r] \ar[d]^{\theta_{W^1\oplus I^1(M)}}& C\otimes_S{X_2}_{*} \ar[d]^{\theta_{X_2}} \ar[r] & 0\\
0 \ar[r] &  X_1 \ar[r] & W^1\oplus I^1(M) \ar[r] &  X_2 \ar[r]& 0.}$$
So $\theta_{X_2}$ is an isomorphism and hence $X_2$ is
2-$C$-cotorsionfree by Corollary 3.4(2). Furthermore, there exists the following
commutative diagram with exact rows:
{\footnotesize $$\xymatrix{ 0 \ar[r] & \Tor^S_1(C,{X_3}_{*}) \ar[r]  & C\otimes_S{X_2}_{*}
\ar[r] \ar[d]^{\theta_{X_2}} & C\otimes_S(W^2\oplus I^2(M))_{*} \ar[r]
\ar[d]^{\theta_{W^2\oplus I^2(M)}}& C\otimes_S{X_3}_{*} \ar[d]^{\theta_{X_3}} \ar[r] & 0\\
& 0 \ar[r] &  X_2 \ar[r] & W^2\oplus I^2(M) \ar[r] &  X_3 \ar[r]& 0.}$$}
So $\theta_{X_3}$ is an isomorphism and $\Tor^S_1(C,{X_3}_{*})=0$, and hence
$X_3$ is 3-$C$-cotorsionfree by Corollary 3.4(3).
Repeating a similar argument, we eventually get that $\coOmega^{n}(M)\cong X_n$ is
$n$-$C$-cotorsionfree.
\end{proof}

The following result is an addendum to Theorem 3.9.

\begin{prop} \label{Prop: 3.10} Let $M\in \Mod R$ and $n\geq 1$. If
$\coOmega^{n}(M)$ is $\infty$-$C$-cotorsionfree, then there exists an
exact sequence $0\rightarrow M\rightarrow X\rightarrow Y\rightarrow
0$ in $\Mod R$ with $X$ $\infty$-cotorsionfree and $\Add_RC$-$\id_R Y\leq n-1$.
\end{prop}

\begin{proof} We proceed by induction on $n$.

Let $n=1$. Since $\coOmega^{1}(M)$ is $\infty$-$C$-cotorsionfree by assumption,
there exists an exact sequence $0\rightarrow N_1\rightarrow
W_1\rightarrow \coOmega^{1}(M)\rightarrow 0$ in $\Mod R$ with $W_1\in\Add_RC$, $N_1$
$\infty$-$C$-cotorsionfree and $\Ext_R^1(C,N_1)=0$
by Proposition 3.6. Consider the following
pullback diagram:
$$\xymatrix{& & 0 \ar[d]& 0 \ar[d] &  \\
& & N_1 \ar@{=}[r] \ar[d] & N_1  \ar[d]& \\
0 \ar[r]&  M \ar[r] \ar@{=}[d]& X \ar[r] \ar[d]& W_1 \ar[r] \ar[d] & 0 \\
0 \ar[r]& M \ar[r] & I^0(M) \ar[r] \ar[d]& \coOmega^{1}(M) \ar[r] \ar[d] & 0 \\
& & 0 & 0. }$$ It follows from Proposition 3.5 that the middle row in
the above diagram is the desired sequence.

Now suppose $n\geq 2$. By the induction hypothesis, there exists an
exact sequence $0\rightarrow \coOmega^{1}(M)\rightarrow
X'\rightarrow Y'\rightarrow 0$ in $\Mod R$ with $X'$ $\infty$-$C$-cotorsionfree and
$\Add_RC$-$\id_RY'\leq n-2$. We
also have an exact sequence $0\rightarrow X''\rightarrow W'\rightarrow
X'\rightarrow 0$ in $\Mod R$ with $W'\in\Add_RC$, $X''$ $\infty$-$C$-cotorsionfree
and $\Ext_R^1(C,X'')=0$  by Proposition 3.6. We have the following pullback diagram:
$$\xymatrix{ & 0 \ar[d]& 0 \ar[d] & \\
& X'' \ar[d] \ar@{=}[r]& X'' \ar[d] & \\
0 \ar[r] & Y \ar[r] \ar[d]& W' \ar[r] \ar[d] & Y' \ar[r] \ar@{=}[d]& 0 \\
0 \ar[r] &  \coOmega^1(M) \ar[r] \ar[d]& X' \ar[r] \ar[d] & Y' \ar[r] &0\\
& 0 & 0. &}$$
Then $\Add_RC$-$\id_R Y\leq n-1$. Consider the following pullback diagram:
$$\xymatrix{& & 0 \ar[d]& 0 \ar[d] & \\
& & X'' \ar[d] \ar@{=}[r]& X'' \ar[d] & \\
0 \ar[r]&  M \ar[r] \ar@{=}[d]& X \ar[r] \ar[d]& Y \ar[r] \ar[d] & 0 \\
0 \ar[r]&  M \ar[r] & I^0(M) \ar[r] \ar[d]& \coOmega^1(M) \ar[r] \ar[d] & 0\\
& & 0 & 0. &}$$ Note that the middle column in this diagram is $\Hom_R(C,-)$-exact.
So $X$ is $\infty$-$C$-cotorsionfree
by Proposition 3.5. Therefore the middle row in this diagram is as desired.
\end{proof}

\section {\bf Cograde and Cotorsionfreeness }

\bigskip

In this section, for a module $M\in \Mod R$ and a positive integer $n$, we will give a criterion
in terms of the properties of the cograde of modules for judging when $\coOmega^{i}(M)$
is $i$-$C$-cotorsionfree for any $1\leq i\leq n$.

Let $M \in \Mod R$ and $n\geq 1$. From the exact sequence:
$$0\rightarrow \coOmega^{n-1}(M)\stackrel{\lambda^{n-1}}{\rightarrow}
I^{n-1}(M)\stackrel{p^n}{\rightarrow}\coOmega^{n}(M)\rightarrow 0 \eqno{(4.1)}$$ we get the following exact sequence:
$$0\rightarrow (\coOmega^{n-1}(M))_*\stackrel{{\lambda^{n-1}}_*}{\rightarrow}
I^{n-1}(M)_*\stackrel{{p^{n}}_*}\rightarrow
(\coOmega^{n}(M))_*\rightarrow \Ext_S^n(C,M)\rightarrow 0.$$
Set $\Im {p^n}_*=N$, and decompose this
sequence into two short exact sequences:
$$0\rightarrow (\coOmega^{n-1}(M))_*\stackrel{{\lambda^{n-1}}_*}{\rightarrow}
I^{n-1}(M)_*\stackrel{\beta}{\rightarrow} N\rightarrow 0$$ and
$$0\rightarrow N\stackrel{\alpha}{\rightarrow}(\coOmega^{n}(M))_*\rightarrow
\Ext_R^n(C,M)\rightarrow 0.\eqno{(4.2)}$$
Then we get the following commutative diagram with exact rows:
$$\xymatrix{ & C\otimes_S(\coOmega^{n-1}(M))_*
\ar[r]^{1_C\otimes {\lambda^{n-1}}_*} \ar[d]^{\theta_{\coOmega^{n-1}(M)}} & C\otimes_SI^{n-1}(M)_*
\ar[r]^{1_C\otimes \beta} \ar[d]^{\theta_{I^{n-1}(M)}}& C\otimes_SN \ar@{-->}[d]^g \ar[r] & 0\\
0 \ar[r] &  \coOmega^{n-1}(M) \ar[r]^{\lambda^{n-1}} & I^{n-1}(M) \ar[r]^{p^{n}} &  \coOmega^{n}(M) \ar[r]& 0.\\
&  & {\rm Diagram}\ (4.1) &  }$$
Then it is straightforward to check that there exists the following commutative diagram with exact rows:
$$\xymatrix{
C\otimes_SN
\ar[d]^g \ar[r]^{1_C\otimes \alpha}& C\otimes_S(\coOmega^{n}(M))_* \ar[r]
\ar[d]^{\theta_{\coOmega^{n}(M)}}& C\otimes_S\Ext^n_R(C,M) \ar[r]& 0\\
\coOmega^{n}(M) \ar@{=}[r]& \coOmega^{n}(M).\\
& {\rm Diagram}\ (4.2)}$$

\begin{lem} \label{lem: 4.1} For a module $M\in \Mod R$, we have
\begin{enumerate}
\item $\coOmega^{1}(M)$ is 1-$C$-cotorsionfree.
\item For any $n\geq 2$, $\Ker \theta_{\coOmega^n(M)} \cong C\otimes_S\Ext_R^n(C,M).$
\end{enumerate}
\end{lem}
\begin{proof}(1). Since $I^0(M)$ is $\infty$-$C$-cotorsionfree by Theorem 3.8,
the assertion follows from Corollary 3.7.

(2). If $n\geq 2$, then $\theta_{\coOmega^{n-1}(M)}$ is an epimorphism by (1).
Because $\theta_{I^{n-1}(M)}$ is an isomorphism by Theorem 3.8, $g$ is an isomorphism in the above two diagrams.
So $\Ker\theta_{\coOmega^{n}(M)} \cong
C\otimes_S\Ext_R^n(C,M)$.
\end{proof}

The notion of the cograde of finitely generated modules was introduced in \cite[Corollary 3.11]{O}. The following definition
generalizes it to a general setting.

\begin{df} \label{df: 4.2}
{\rm  For a module $N\in \Mod S$, the \emph{cograde} of $N$ with
respect to $C$ is defined by $\cograde_CN:=\inf\{i\mid\Tor^S_i(C,N)
\neq 0\}$.}
\end{df}

We are now in a position to give the main result in this section,
which can be regarded as a dual version of \cite[Proposition 2.26]{Au}.

\begin{thm} \label{thm: 4.3} Let $M\in \Mod R$ and $n\geq 1$.
Then $\coOmega^{i}(M)$ is $i$-$C$-cotorsionfree for any $1\leq i\leq n$ if
and only if $\cograde_C\Ext_R^i(C,M)\geq i-1$ for any $1\leq i\leq n$.
\end{thm}

\begin{proof} We proceed by induction on $n$.
If $n=1$, then the assertion follows from Lemma 4.1(1).

Let $n=2$. Then $\coOmega^{2}(M)$ is $2$-$C$-cotorsionfree if and only if
$\theta_{\coOmega^{2}(M)}$ is an isomorphism by Corollary 3.4(2). Note that
$\theta_{\coOmega^{2}(M)}$ is epic by Lemma 4.1(1). So
$\coOmega^{2}(M)$ is $2$-$C$-cotorsionfree if and only if
$\theta_{\coOmega^{2}(M)}$ is monic. But $\Ker\theta_{\coOmega^{2}(M)} \cong
C\otimes_S\Ext_R^2(C,M)$ by Lemma 4.1(2). So
$\coOmega^{2}(M)$ is $2$-$C$-cotorsionfree if and only if
$C\otimes_S\Ext_S^2(C,M)=0$, that is, $\cograde _C\Ext_S^2(C,M)\geq 1$.

Now suppose $n\geq 3$.

If $\coOmega^{i}(M)$ is $i$-$C$-cotorsionfree for
$1\leq i\leq n$, then by the induction hypothesis, it suffices to show that
$\cograde_C\Ext_S^n(C,M)\geq n-1$. By Lemma 4.1(2), $C\otimes_S\Ext_R^n(C,M)\cong \Ker \theta_{\coOmega^n(M)}=0$.
From the exact sequence (4.1), we get the following exact sequence:
$$\Tor_1^S(C,(\coOmega^{n}(M))_*)\rightarrow \Tor_1^S(C,\Ext^n_R(C,M))$$
$$\rightarrow C\otimes_SN
\stackrel{1_C\otimes \alpha}{\rightarrow} C\otimes_S(\coOmega^{n}(M))_* \rightarrow C\otimes_S\Ext^n_R(C,M) \rightarrow 0.$$
Because both $\theta_{\coOmega^{n-1}(M)}$ and $\theta_{I^{n-1}(M)}$ are isomorphisms, the homomorphism
$g$ in the diagram behind (4.2) is also an isomorphism. Then from Diagram (4.2) we know that $1_C\otimes \alpha$
is monic. By Corollary 3.4(3) we have
$\Tor_i^S(C,(\coOmega^{n}(M))_*)=0$ for any $1\leq i\leq n-2$. So $\Tor_1^S(C,\Ext^n_R(C,M))=0$, and hence
$\cograde_C\Ext_R^n(C,M)\geq 2$.

From the exact sequence (4.1) get the following exact sequence:
$$0\rightarrow (\coOmega^{n-1}(M))_*\stackrel{{\lambda^{n-1}}_*}{\rightarrow}
I^{n-1}(M)_*\stackrel{{p^{n}}_*}\rightarrow
(\coOmega^{n}(M))_*\rightarrow \Ext_S^n(C,M)\rightarrow 0.\eqno{(4.3)}$$
By Theorem 3.8 and Corollary 3.4(3), $\Tor_i^S(C,I^{n-1}(M)_*)=0$ for any $i\geq 1$.
Again by Corollary 3.4(3) we have
$\Tor_i^S(C,(\coOmega^{n-1}(M))_*)=0$ for any $1\leq i\leq n-3$. So by the dimension shifting,
$\Tor_i^S(C,\Ext^n_R(C,M))=0$ for any $3\leq i\leq n-2$, and hence
$\cograde_C\Ext_R^n(C,M)\geq n-1$.

Conversely, if $\cograde_C\Ext_R^i(C,M)\geq i-1$ for any $1\leq i\leq n$,
then by the induction hypothesis, it suffices to show that
$\coOmega^{n}(M)$ is $n$-$C$-cotorsionfree. Since
$\coOmega^{n-1}(M)$ is $(n-1)$-$C$-cotorsionfree by the induction hypothesis, $\theta_{\coOmega^{n-1}(M)}$
is an isomorphism. Notice that $\theta_{I^{n-1}(M)}$ is also an isomorphism, so is the homomorphism
$g$ in Diagram (4.1). Because $\cograde_C\Ext_R^n(C,M)\geq n-1$ by assumption,
$1_C\otimes \alpha$ in Diagram (4.2) is an isomorphism. It implies that
$\theta_{\coOmega^{n}(M)}$ is also an isomorphism and $\coOmega^{n}(M)$ is $C$-coreflexive.
On the other hand, similar to the above argument, using the dimension shifting, from the exact sequence (4.3) we get that
$\Tor_i^S(C,(\coOmega^{n}(M))_*)=0$ for any $1\leq i\leq n-2$. Then we conclude that $\coOmega^{n}(M)$ is $n$-$C$-cotorsionfree
by Corollary 3.4(3).
\end{proof}

\section {\bf Special cotorsionfree modules over artin algebras}

\bigskip

Throughout this section, $\Lambda$ is an artin $R$-algebra over a
commutative artin ring $R$. Let $\mod \Lambda$ be the class of
finitely generated left $\Lambda$-modules. We denote by $D$ the ordinary Matlis duality
between $\mod\Lambda^{op}$ and $\mod \Lambda$, that is, $D(-):=\Hom_R(-,I^0(R/J(R)))$, where $J(R)$ is the
Jacobson radical of $R$ and $I^0(R/J(R))$ is the injective envelope
of $R/J(R)$. It is easy to verify that ($\Lambda, \Lambda$)-bimodule
$D(\Lambda)$ is semidualizing. We use $\add D(\Lambda)$ to denote
the subclass of $\mod \Lambda$ consisting of modules isomorphic to
direct summands of finite direct sums of copies of $D(\Lambda)$. We
use abbreviation $\cTr(-)$ for $\cTr_{D(\Lambda)}(-)$. Let $A\in
\mod\Lambda$ and $n\geq 1$. Then $A$ is called
\emph{$n$-cotorsionfree} if $\Tor_i^{\Lambda}(D(\Lambda),\cTr A)=0$
for any $1\leq i\leq n$, and $A$ is called
\emph{$\infty$-cotorsionfree} if it is $n$-cotorsionfree for all
$n$; in particular, every module in $\mod \Lambda$ is
0-cotorsionfree. In addition, $A$ is called \emph{$n$-cospherical}
if $\Ext_{\Lambda}^i(D(\Lambda),A)=0$ for any $1\leq i\leq n$, and
$A$ is called \emph{$\infty$-cospherical} if it is $n$-cospherical
for all $n$.

Put $(-)^*=:\Hom_{\Lambda}(-,\Lambda)$. The following result establishes the dual relation between
the cotranspose (resp. $n$-cotorsionfree modules) and the transpose (resp. $n$-torsionfree modules).

\begin{prop} \label{prop: 5.1} Let $A\in\mod \Lambda$ and $n\geq 1$. Then we have
\begin{enumerate}
\item $\Tr A\cong\cTr D(A)$.
\item $\cTr A \cong\Tr D(A)$.
\item $A$ is $n$-torsionfree if and only if $D(A)$ is $n$-cotorsionfree.
\item $A$ is $n$-cotorsionfree if and only if $D(A)$ is $n$-torsionfree.
\end{enumerate}
\end{prop}

\begin{proof} Because (2) and (4) are duals of (1) and (3) respectively, it suffices to prove (1) and (3).

(1) Let
$$P_1\rightarrow P_0\rightarrow A\rightarrow 0$$ be a minimal projective presentation of $A$ in $\mod \Lambda$.
Then we have the following exact sequence:
$$0\rightarrow A^*\rightarrow P_0^{*}\rightarrow P_1^{*}\rightarrow
\Tr A\rightarrow 0,$$ and a minimal injective presentation:
$$0\rightarrow D(A)\rightarrow D(P_0)\rightarrow
D(P_1)$$ of $D(A)$.  Now we obtain another exact sequence:
{\footnotesize$$0\rightarrow \Hom_{\Lambda}(D(\Lambda),D(A))\rightarrow
\Hom_{\Lambda}(D(\Lambda),D(P_0))\rightarrow
\Hom_{\Lambda}(D(\Lambda),D(P_1))\rightarrow\cTr D(A)\rightarrow
0.$$} Since $P_i^*\cong\Hom_{\Lambda}(D(\Lambda),D(P_i))$ for
$i=1,2$, $\Tr A\cong\cTr D(A)$.

(3) For any $i\geq 1$, we have \begin{align*}
&\ \ \ \ \Ext_{\Lambda}^i(\Tr A,\Lambda)\\
& \cong\Ext_{\Lambda}^i(\Tr A,\Hom_{\Lambda}(D(\Lambda),D(\Lambda)))\\
& \cong\Hom_{\Lambda}(\Tor_i^{\Lambda}(\Tr A,D(\Lambda)),D(\Lambda)) \ \text{(by \cite[Chapter VI, Proposition 5.1]{CE})}.
\end{align*}
Note that $D(\Lambda)$ is an injective cogenerator for $\Mod \Lambda$.
So, for any $i\geq 1$ we have that $\Ext_{\Lambda}^i(\Tr A,\Lambda)=0$ if and only if $\Tor_i^{\Lambda}(\Tr A,D(\Lambda))=0$,
and if and only if $\Tor_i^{\Lambda}(\cTr D(A),D(\Lambda))=0$ by Proposition 5.1(1). It follows that
$A$ is $n$-torsionfree if and only if $D(A)$ is $n$-cotorsionfree.
\end{proof}

Note that a module in $\mod \Lambda^{op}$ is Gorenstein flat (see \cite{EJ} for the definition) if and only if it
is Gorenstein projective by \cite[Proposition 1.3]{BM}. So the ordinary Matlis duality $D$ between $\mod \Lambda^{op}$
and $\mod \Lambda$ induces a duality between
Gorenstein projective modules in $\mod \Lambda^{op}$ and Gorenstein injective modules in $\mod \Lambda$ (c.f. \cite[Theorem 3.6]{H}).
Then by \cite[Proposition 10.2.6]{EJ} and Proposition 5.1, we immediately have the following

\begin{cor} \label{cor: 5.2}  For a module $A\in \mod \Lambda$, the following statements are equivalent.
\begin{enumerate}
\item $A$ is $\infty$-cotorsionfree and $\infty$-cospherical.
\item There exists a totally $\add D(\Lambda)$-acyclic complex
$\bf{I}$ (as in Definition 2.7) such that $A\cong \Im(I_0\rightarrow I^{0}).$
\item $A$ is Gorenstein injective.
\end{enumerate}
\end{cor}

Recall that $\Lambda$ is called \emph{Gorenstein} if $\id_{\Lambda} \Lambda=\id_{\Lambda^{op}}\Lambda<\infty$.

\begin{cor} \label{cor: 5.3} The following statements are
equivalent for any $n\geq 0$.
\begin{enumerate}
\item $\Lambda$ is Gorenstein with $\id_{\Lambda} \Lambda=\id_{\Lambda^{op}}\Lambda\leq n$.
\item The $n$-cosyzygy of a module in $\mod \Lambda$ and that of a module in $\mod \Lambda^{op}$ are
$\infty$-cotorsionfree.
\item Every module in $\mod \Lambda$ and every module in $\mod \Lambda^{op}$ are quotient modules of
a left $\Lambda$-module and a right $\Lambda$-module with injective dimension at most $n$ respectively.
\end{enumerate}
\end{cor}

\begin{proof} By Corollary 5.2 and \cite[Theorem 1.4 and Lemma 3.8]{HH}, using the duality functor $D$
we get the assertion. \end{proof}




The following example illustrates that the condition
``$\infty$-cotorsionfree" in Corollary 5.3(2) can not be replaced by ``$n$-cotorsionfree".

\begin{exa} \label{exa: 5.5} Let $\Lambda$ be a finite-dimensional algebra
over an algebraically closed field given by the quiver:
\[\xymatrix{
\cdot \ar@(dl,ul)[]^{\alpha}
\ar@(dr,ur)[]_{\beta}   }\]modulo the ideal generated by $\{\alpha^2, \beta^2, \alpha\beta,
\beta\alpha\}$. Then $\Lambda$ is not Gorenstein, but for any $A\in\mod \Lambda$, $\coOmega^{1}(A)$ is
1-cotorsionfree.
\end{exa}

\begin{cor} \label{cor: 5.5} If both $R$ and $\Lambda$ are
local, then the following statements are equivalent.
\begin{enumerate}
\item $\Lambda$ is Gorenstein.
\item $\Lambda$ is self-injective.
\item For $A\in\mod \Lambda$ and $B\in\mod \Lambda^{op}$, $D(\Lambda)\otimes_{\Lambda}\cTr A$ and $\cTr B\otimes_{\Lambda}D(\Lambda)$ are Gorenstein injective.
\item For $A\in\mod \Lambda$ and $B\in\mod \Lambda^{op}$, $D(\Lambda)\otimes_{\Lambda}\cTr A$ and $\cTr B\otimes_{\Lambda}D(\Lambda)$ are $\infty$-cotorsionfree.
\end{enumerate}
\end{cor}
\begin{proof} $(1) \Rightarrow (2)$ follows from \cite[Corollary 2.15]{Ra}, and
$(3)\Rightarrow (4)$ follows from Corollary 5.2.

Note that $D(\Lambda)\otimes_{\Lambda}\cTr A$
(resp. $\cTr B\otimes_{\Lambda}D(\Lambda)$) is isomorphic to the $2$-cosyzygy
of $A$ (resp. $B$). So both $(4) \Rightarrow (1)$ and $(2)\Rightarrow (3)$ follow from Corollaries 5.3 and 5.2.
\end{proof}

Let $\mathcal {GI}$ denote the class of finitely generated Gorenstein injective left
$\Lambda$-modules. We write $^{\bot}\mathcal{GI}=\{M\in
\mod \Lambda\mid \Ext_{\Lambda}^{\geq 1}(M,X)=0$ for any $X\in \mathcal{GI}\}$
and $(^{\bot}\mathcal{GI})^{\bot}=\{M\in \mod \Lambda\mid
\Ext_{\Lambda}^{\geq 1}(M,Y)=0$ for any $Y\in {^{\bot}\mathcal{GI}}\}$.

\begin{lem} \label{lem: 5.6} Let $0\rightarrow L\rightarrow M\rightarrow N\rightarrow 0$
be an exact sequence in $\mod \Lambda$.
If $L, M\in{^\bot\mathcal{GI}}$, then $N \in{^\bot\mathcal{GI}}$.
\end{lem}
\begin{proof} By the dimension shifting, we have $\Ext_{\Lambda}^{\geq 2}(N,A)=0$ for any $A\in \mathcal{GI}$.
Now let $X\in \mathcal{GI}$. It suffices to prove $\Ext_{\Lambda}^1(N,X)=0$.
By Corollary 5.2 there exists an exact sequence $0\rightarrow K\rightarrow I_0\rightarrow X\rightarrow 0$ in $\mod \Lambda$ with
$I_0\in \add D(\Lambda)$ and $K\in \mathcal{GI}$. So
$\Ext_{\Lambda}^1(N,X)\cong \Ext_{\Lambda}^2(N,K)=0$.
\end{proof}

Let $\mathcal{X}$ be a full subcategory of an abelian category $\mathcal{A}$.
We write $^\bot\mathcal{X}=\{M\in \mathcal{A}\mid
\Ext_{\mathcal{A}}^{\geq 1}(M,X)=0$ for any $X\in \mathcal{X}\}$ and $\mathcal{X}^\bot=\{M\in
\mathcal{A}\mid \Ext_{\mathcal{A}}^{\geq 1}(X,M)=0$ for any $X\in \mathcal{X}\}$.
Recall that a pair of subcategories $(\mathcal{X},\mathcal{Y})$ of an abelian category $\mathcal{A}$ is
called a \emph{cotorsion pair} if $\mathcal{X}={^{\bot}\mathcal{Y}}$ and $\mathcal{Y}={\mathcal{X}^{\bot}}$.
We denote by $\GInj(\Lambda)$ the subclass of $\Mod \Lambda$ consisting of Gorenstein injective modules,
and write $^\bot\GInj(\Lambda)=\{M\in \Mod \Lambda\mid
\Ext_{\Lambda}^{\geq 1}(M,X)=0$ for any $X\in \GInj(\Lambda)\}$. It is known that ($^\bot\GInj(\Lambda)$, $\GInj(\Lambda)$)
forms a cotorsion pair in $\Mod \Lambda$ (\cite{BM}).
The following result gives an equivalent characterization when $({^\bot\mathcal{GI}},\mathcal{GI})$ forms a cotorsion pair
in $\mod \Lambda$.

\begin{thm} \label{thm: 5.7} The following statements are
equivalent.
\begin{enumerate}
\item ${(^\bot\mathcal{GI})^\bot}=\mathcal{GI}$ (that is, $({^\bot\mathcal{GI}},\mathcal{GI})$ forms a cotorsion pair).
\item Every module in $(^\bot\mathcal{GI})^\bot$ is 1-cotorsionfree.
\item Every module in $(^\bot\mathcal{GI})^\bot$ is $\infty$-cotorsionfree.
\end{enumerate}
\end{thm}
\begin{proof} $(1)\Rightarrow (2)$ follows from Corollary 5.2.

$(2)\Rightarrow (3)$ Let $A\in (^\bot\mathcal{GI})^\bot$. Then $A$ is
1-cotorsionfree by assumption. So there exists a $\Hom_{\Lambda}(\add D(\Lambda),-)$-exact exact sequence
$0\rightarrow K\rightarrow I_{0}\rightarrow A\rightarrow 0$ in $\mod \Lambda$ with
$I_0\in \add D(\Lambda)$ by Proposition 3.6. We claim that $K\in(^\bot\mathcal{GI})^\bot$. Let
$Y\in{^\bot\mathcal{GI}}$. Then $\Ext_{\Lambda}^i(Y,K)\cong
\Ext_{\Lambda}^{i-1}(Y,A)$ for any $i\geq 2$. Note that $\coOmega^{1}(Y)\in{^\bot\mathcal{GI}}$
by Lemma 5.6. Then from the exact sequence
$\Ext^1_{\Lambda}(I^0(Y),K)\rightarrow
\Ext^1_{\Lambda}(Y,K)\rightarrow
\Ext^2_{\Lambda}(\coOmega^{1}(Y),K)$ we get that
$\Ext_{\Lambda}^1(Y,K)=0$. The claim follows. So $K$ is 1-cotorsionfree by assumption, and
hence $A$ is 2-cotorsionfree by Proposition 3.6. By replacing $A$ by $K$ in the above argument, we get that
$K$ is 2-cotorsionfree and then $A$ is 3-cotorsionfree. Continuing this process, we finally have that
$A$ is $\infty$-cotorsionfree.

$(3)\Rightarrow (1)$ Obviously $\mathcal{GI}\subseteq(^\bot\mathcal{GI})^\bot$.
Now let $A\in (^\bot\mathcal{GI})^\bot$. It suffices to prove that $A$ is Gorenstein injective.
Because $D(\Lambda)\in{^\bot\mathcal{GI}}$,
$\Ext^{\geq 1}_{\Lambda}(D(\Lambda),A)=0$ and $A$ is $\infty$-cospherical. Note that
$A$ is $\infty$-cotorsionfree by assumption.
It follows from Corollary 5.2 that $A$ is Gorenstein injective.
\end{proof}

\begin{prop} \label{prop: 5.8} Let $R$ be a commutative local artin ring and $F$ a free $R$-module with $\rank(F)=2n$.
If there exists an endomorphism $f$ of $F$ such that $f^2=0$ and
$\rank (\Im f)=\rank(\Im f^*)=n$, then $(\Im f)^\vee$ is
$\infty$-cotorsionfree, where $(-)^\vee=\Hom_R(-,I^0(R/J(R)))$.
\end{prop}

\begin{proof} Since $f^2=0$, there exists a complex $0\rightarrow \Im f
\rightarrow F\stackrel{f}{\rightarrow}F\stackrel{f}{\rightarrow}\cdots$.
Now consider the short exact sequence $0\rightarrow \Ker f
\rightarrow F\rightarrow \Im f\rightarrow 0$. Because rank($\Im f)=
\rank(F)/2$ by assumption, $\rank(\Im f)=\rank(\Ker f)$. Observing that $\Im f
\subseteq\Ker f$, so $\Im f=\Ker f$. Thus the above complex is exact.
In a similar way, we also get that $\Im f^*=\Ker f^*$. Hence $\Im f$
is $\infty$-torsionfree by \cite[Theorem 2.17]{Au}. So $(\Im f)^\vee$ is $\infty$-cotorsionfree
by Proposition 5.1.
\end{proof}

We give an example to illustrate Proposition 5.8.

\begin{exa} \label{exa: 5.9} Let $k$ be a field and $S=k[[X]]$ and $R=S/(X^2)$, and let $F=
R^2$ and $f:$ $R^2\rightarrow R^2$ a map given by the matrix:
\begin{equation}
\left(
\begin{array}{cccc}
x &  0 \\
x &  x \\
\end{array}
\right).
\end{equation}
Then $(\Im f)^\vee$ is a non-injective $\infty$-cotorsionfree module.
\end{exa}
\begin{proof}
$R$ has a basis consisting of the following 2 elements: 1, $x$,
where $x$ denotes the residue class of the variable $X$ modulo the
ideal $<X^2>$. It is easy to check that $\rank(F)=4$ and $f^2=0$.
Since $\Im f$ is generated by the elements: $f(1,0)=(x,x)$,
$f(0,1)=(0,x)$. It is clear that $\rank(\Im f)=2$. Similarly, the
map $f^*$ is given by the transpose of the matrix defining $f$. One
can see that $\rank(\Im f^*)=2$. Notice that $\Im f$ is not
isomorphic to a direct summand of $R^2$. So $\Im f$ is not
projective. Consequently one gets the assertion by Proposition 5.8.
\end{proof}

Huang and Huang raised in \cite{HH} an open question: Is the class of $\infty$-torsionfree modules closed
under kernels of epimorphisms? We will give an example to show that for any $n\geq 2$, neither the class of
$n$-torsionfree modules nor that $\infty$-torsionfree modules is
closed under kernels of epimorphisms in general. Nevertheless, the class of
$1$-torsionfree modules is closed under kernels of epimorphisms,
since every submodule of a 1-torsionfree module is also
1-torsionfree. The following example is due to Jorgensen and \c{S}ega (see \cite{JS}).

\begin{exa} \label{exa: 5.10} Suppose that $R=\mathbb{Q}[V,X,Y,Z]/I$, where $\mathbb{Q}$
is the field of rational numbers and $I$ =
$<V^2,Z^2,XY,VX+2XZ,VY+YZ,VX+Y^2,VY-X^2>$. Let $f: R^2\rightarrow
R^2$ denote the map given by the matrix: \begin{equation} \left(
\begin{array}{cccc}
v &  2x \\
y & z \\
\end{array}
\right),
\end{equation} where $v,x,y,z$ denote the residue classes of the
variables modulo $I$. Take $M=\Coker f$ and $N=\Im f$. Then
there exists an exact sequence $0\rightarrow N\rightarrow R^2\rightarrow
M\rightarrow 0$ such that $M$ is $\infty$-torsionfree and $N$ is not
$n$-torsionfree for any $n\geq 2$.
\end{exa}
\begin{proof} From \cite[Lemma 1.5]{JS} we know that $M$ is
$\infty$-torsionfree. By \cite[Lemma 1.4]{JS} we have a free
presentation $R^2\stackrel{g}{\rightarrow} R^2\rightarrow N\rightarrow
0$ of $N$, where $g$ is given by the following matrix:
\begin{equation} \left(
\begin{array}{cccc}
v & x \\
y & z \\
\end{array}
\right).
\end{equation} Then $\Im g^{*}$ is generated by the following elements:

\indent $g^{*}(1,0)=(v,x)$ $\,\,\,\,\,\,\,\,\,$
$g^{*}(z,0)=(vz,-2^{-1}vx)$ \\
\indent $g^{*}(0,1)=(y,z)$ $\,\,\,\,\,\,\,\,\,$
$g^{*}(0,v)=(vy,vz)$ \\
\indent $g^{*}(v,0)=(0,vx)$ $\,\,\,\,\,\,$
$g^{*}(0,x)=(0,-2^{-1}vx)$ \\
\indent $g^{*}(x,0)=(vx,vy)$ $\,\,$
$g^{*}(0,y)=(-vx,-vy)$ \\
\indent $g^{*}(y,0)=(vy,0)$ $\,\,\,\,\,\,$
$g^{*}(0,z)=(-vy,0)$ \\
One can use a computer algebra software, like Singular (see
\cite{GP}), to verify that $\Ext_R^1(\Im g^{*},R)\neq 0$. Thus
$\Ext_R^2(\Tr N,R)\neq 0$, and therefore $N$ is not $n$-torsionfree
for any $n\geq 2$. The computation of $\Ext_R^1(\Im g^{*},R)$ by
Singular is as follows.

LIB "homolog.lib";

ring $S=0, (V,X,Y,Z)$, dp;

ideal $I=V2, Z2, XY, VX+2XZ, VY+YZ, VX+Y2, VY-X2$;

qring $R=\std(I)$; // define the ring $R$

module $F = [V,X], [2VZ,-VX], [Y,Z], [VY,VZ], [0,VX], [VX,VY],
[VY,0]$;

module $H=1$;

module $E=\Ext(1, \syz(F), \syz(H));$ // compute $\Ext_R^1(\Im
g^{*},R)$

\noindent The output says that the dimension of $\Ext_R^1(\Im
g^{*},R)$ as a vector space is 3.
\end{proof}
By Example 5.10 and Proposition  5.1, we have that the
class of $\infty$-cotorsionfree modules is not closed under
cokernels of monomorphisms in general.

\vspace{0.5cm}

{\bf Acknowledgements.} This research was partially supported by NSFC (Grant No. 11171142), Mathematical Tianyuan Fund
of NSFC (Grant No. 11226058) and NSF of Guangxi Province of China (Grant No. 2013GXNSFBA019005).

\end{document}